\documentclass[12pt]{amsart}
\usepackage{latexsym,epsfig,amssymb,amsmath,amsthm,color,url,pdfsync,bbm}
\usepackage[square,comma,sort&compress]{natbib}
\setcitestyle{numbers,multirow,import}
\usepackage{graphicx}
\usepackage{inputenc}

\allowdisplaybreaks
\numberwithin{equation}{section}

\newtheorem{Theorem}{Theorem}[section]
\newtheorem{Lemma}[Theorem]{Lemma}
\newtheorem{Remark}[Theorem]{Remark}
\newtheorem{Example}[Theorem]{Example}
\newtheorem{Proposition}[Theorem]{Proposition}
\newtheorem{Definition}[Theorem]{Definition}
\newtheorem{Corollary}[Theorem]{Corollary}
\usepackage{enumerate}

\def\Mr{M^{(r)}}
\def\Yr{Y^{(r)}}
\def\bYr{\boldsymbol{Y}^{(r)}}
\def\bMr{\boldsymbol M^{(r)}}
\def\Lr{L^{(r)}}
\def\Tr{T^{(r)}}
\def\gammai{\gamma^{(i)}}
\def\Z{\mathbb{Z}}
\def\eqd{\stackrel{d}{=}}
\def\nuhatstand{\hat\nu_{\text{standard},n}}
\def\cond{\,|\,}
\def\Starica{St\u{a}ric\u{a} }
\def\Rootzen{Rootz\'en }
\def\SaS{S$\alpha$S\ }
\def\one{{\bf 1}}
\def\Fbar{{\overline F}}
\def\lst{\buildrel{\rm st}\over{\leqs}}
\def\Leb{\text{{\it Leb}}}
\def\essot{s_T^{(0)} }
\def\supp{\text{supp} }
\def\Aff{\text{Aff} }
\def\bX{\boldsymbol X}
\def\bY{\boldsymbol Y}
\def\bU{\boldsymbol U}
\def\bZ{\boldsymbol Z}
\def\bu{\boldsymbol u}
\def\bone{\boldsymbol 1}
\def\ba{\boldsymbol a}
\def\boldb{\boldsymbol b}
\def\bq{\boldsymbol q}
\def\bs{\boldsymbol s}
\def\bt{\boldsymbol t}
\def\bzero{\boldsymbol 0}
\def\be{\boldsymbol e}
\def\bx{\boldsymbol x}
\def\bz{\boldsymbol z}
\def\boldb{\boldsymbol b}
\def\by{\boldsymbol y}
\def\bj{\boldsymbol j}
\def\bJ{\boldsymbol J}
\def\binfty{\boldsymbol \infty}
\def\bTheta{\boldsymbol \Theta}
\def\balpha{\boldsymbol \alpha}
\def\btheta{\boldsymbol \theta}
\def\E{\mathbb{E}}
\def\R{\mathbb{R}}
\def\C{\mathfrak{C}}
\def\V{\boldsymbol V}

\def\max{\vee}

\def\cadlag{c\`adl\`ag}

\newcommand{\bthe}{\begin{Theorem}}
\newcommand{\ethe}{\end{Theorem}}

\newcommand{\ble}{\begin{Lemma}}
\newcommand{\ele}{\end{Lemma}}

\newcommand{\bde}{\begin{Definition}}
\newcommand{\ede}{\end{Definition}}

\newcommand{\bco}{\begin{Corollary}}
\newcommand{\eco}{\end{Corollary}}

\newcommand{\bpr}{\begin{Proposition}}
\newcommand{\epr}{\end{Proposition}}

\newcommand{\brem}{\begin{Remark}}
\newcommand{\erem}{\end{Remark}}

\newcommand{\bexam}{\begin{Example}}
\newcommand{\eexam}{\end{Example}}

\newcommand{\beqq}{\begin{equation}}
\newcommand{\eeqq}{\end{equation}}

\newcommand{\beao}{\begin{eqnarray*}}
\newcommand{\eeao}{\end{eqnarray*}\noindent}

\newcommand{\beam}{\begin{eqnarray}}
\newcommand{\eeam}{\end{eqnarray}\noindent}

\newcommand{\barr}{\begin{array}}
\newcommand{\earr}{\end{array}}

\newcommand{\bproof}{\begin{proof}}
\newcommand{\eproof}{\end{proof}}

\newcommand{\sid}[1]{{\color{black} #1}}
\newcommand\independent{\protect\mathpalette{\protect\independenT}{\perp}}
\def\independenT#1#2{\mathrel{\rlap{$#1#2$}\mkern2mu{#1#2}}}
\begin{document}

\bibliographystyle{plain}

\title[$r$th Largest]{Processes of $r$th Largest}

\author[B. Buchmann]{Boris Buchmann}
\address{Dr. Boris Buchmann\\Research School of Finance,
  Actuarial Studies \&  Statistics,\\ 
  Australian National University, ACT 0200, Australia}
\email{Boris.Buchmann@anu.edu.au}

\author[R. Maller]{Ross Maller}
\address{Prof. Ross Maller,\\
Research School of Finance, Actuarial Studies \&  Statistics,\\
Australian National University, Canberra, ACT,Australia,}
\email{Ross.Maller@anu.edu.au}

\author[S.I.\ Resnick]{Sidney I.\ Resnick}
\address{Prof. Sidney Resnick\\
School of Operations Research and Information Engineering\\
Cornell University \\
Ithaca, NY 14853 USA}
\email{sir1@cornell.edu}

\thanks{This
research was initiated and partially supported by ARC grants DP1092502
and DP160104737. S. Resnick also received significant 
support from US Army MURI grant
  W911NF-12-1-0385 to Cornell University; Resnick gratefully acknowledges 
  hospitality, administrative support and space during several visits to the Research School
  of Finance, Actuarial Studies \&  Statistics,
Australian National University.} 




\begin{abstract} {For integers $n\geq r$, we treat the $r$th largest
    of a sample of size $n$ as an $\R^\infty$-valued stochastic
    process in $r$ which we  denote $\bMr$. 
    We show that the sequence regarded in this way
    satisfies the Markov property. We go on to study the asymptotic behaviour of $\bMr$ as $r\to\infty$, and, borrowing from classical 
extreme value theory, show that  left-tail domain of
    attraction conditions  on the underlying distribution of the
    sample guarantee  weak limits for both the range
    of $\bMr$ and $\bMr$ itself, after norming and centering.  
 In continuous time, an analogous
    process $\bYr$ based on a two-dimensional Poisson process on
    $\R_+\times \R$ is treated similarly, but we find that the
    continuous time problems have a distinctive additional feature:
   there are always infinitely many points below the $r$th highest point up to time $t$ for any $t>0$. This necessitates a different approach to the asymptotics in this case. 
}
\end{abstract}
\maketitle

\section{Introduction}\label{sec:intro}
In this paper we consider Markovian and other properties of the order statistics of iid random variables in discrete time, and of extremal processes in continuous time. Although venerable these are important issues and research continues to throw up significant new aspects. 
As a starting point let $\Mr_n$ be  the $r$th largest among iid random variables $X_1,\dots,X_n$ with cdf $F$. (Precise specifications of the order statistics will be given later.) It is known that the {\it finite} sequence $(\Mr_n)_{r=1,2,\ldots,n}$ is Markov if and only if 
$F$ is continuous on $(\ell_F,r_F)$, where  $\ell_F$ and $r_F$ are the left and right extremes of $F$ (see 
\citep{arnold:becker:gather:zahedi:1984}). This is a result concerning 
 the first $r$ order statistics. We proceed to investigate the infinitely many order statistics 
$(\Mr_n, n \geq r)$ beyond the $r$th, and further, derive properties of the whole collection $\{\bMr=(\Mr_n, n \geq r), r\geq 1\}$, considered as an $\R^\infty$-valued stochastic process. Apart from their intrinsic interest the properties we derive bring together 
a number of areas and techniques, as we discuss later.

Thus,  we begin in Section \ref{sec:Markov} by  setting up the notation required for, then proving, the Markovian property, that the conditional distribution of the infinite sequence 
$(M_{r+1}^{(r+1)}, M_{r+2}^{(r+1)}, \ldots)$,
 knowing all values
 $(M_1^{(1)}, M_2^{(1)}, \ldots)$,
 $(M_2^{(2)}, M_3^{(2)}, \ldots)$, $\ldots$, 
 $(M_{r}^{(r)}, M_{r+1}^{(r)}, \ldots)$,
 is the same as  the
conditional distribution  knowing only
 $(M_{r}^{(r)}, M_{r+1}^{(r)}, \ldots)$.
 No continuity assumptions on $F$ are required for this.
 
In Section \ref{sec:asymp} we turn to an investigation of asymptotic properties of the  collection $\bMr$,
for large values of $r$. The weak convergence 
of $\bMr$, after norming and centering,  is related
to domain of attraction theory for the {\it minimum} of an iid sequence of rvs.
A key tool in these proofs is  Ignatov's \citep{ignatov:1977}  theorem showing that 
the $r$-records of an iid sequence are points of a Poisson random measure.

This study is continued in Section \ref {sec:cont_time} for continuous
time $r$th-order extremal processes. Some notable differences between
the discrete and continuous time situations emerge here.  In
particular, unlike in the discrete case, in the continuous time case
there are always infinitely many points below the currently considered
order statistic, and thus the convergence criterion has to be
modified.  Section \ref{sec:final} concludes the paper with some
modest final thoughts and open problems.

We conclude the present section by mentioning previous and related
work.  For alternative proofs and other background on Ignatov's (1977)
theorem see
\citep{ignatov:1977,stam:1985,goldie:rogers:1984,engelen:tommassen:vervaat:1988,
  resnickbook:2008}.  Other treatments of the Markov structure of the
finite sequence $(\Mr_n)_{r=1,2,\ldots,n}$ are in 
\citep{goldie:maller:1999},
\citep{ruschendorf:1985} and 
\citep{cramer:tran:2009}.
The latter two papers show that $(\Mr_n)_{r=1,2,\ldots,n}$ is Markov if information on tied values is incorporated into the sequence.
For background on continuous time extremal processes we refer to
\citep{resnick:1974,resnick:1975,resnickbook:2008,resnick:rubinovitch:1973}.
Additional references are given throughout the text.



\renewcommand{\L}{\mathcal{L}}
\def \K{\mathcal K}
\def \V{\mathcal V}
\def \cS{\mathcal S}
\def \cJ{\mathcal J}
\newcommand{\ov}{\overline}
\newcommand{\pibar}{\overline{\Pi}}
\newcommand{\wtilde}{\widetilde}

\newcommand{\abs}[1]{\left| #1 \right|}
\newcommand{\example}{\quad \emph{Example. }}
\newcommand{\examples}{\quad \emph{Examples. }}
\newcommand{\e}{\mathbf{E}} \newcommand{\p}{\mathbf{P}}
\newcommand{\w}[1]{\widehat{#1}}
\newcommand{\des}{\displaystyle}
\newcommand{\DoubleR}{\rm {I\ \nexto R}}
\newcommand{\todr}{\stackrel{\mathrm{D}}{\longrightarrow}}
\newcommand{\eqdr}{\stackrel{\mathrm{D}}}
\newcommand{\topr}{\stackrel{\mathrm{P}}{\longrightarrow}}
\newcommand{\veps}{\varepsilon}
\newcommand{\rmd}{{\rm d}}
\newcommand{\rmi}{{\rm i}}
\newcommand{\LL}{L\'evy }

\newcommand{\B}{{\bf B:}}
\newcommand{\MDA}{{\mbox{\rm MDA}}}
\newcommand{\DA}{{\mbox{\rm DA}}}
\newcommand{\stp}{\stackrel{P}{\rightarrow}}
\newcommand{\std}{\stackrel{d}{\rightarrow}}
\newcommand{\stdspace}[1]{\stackrel{d}{\rightarrow}_{#1}}
\newcommand{\stas}{\stackrel{\rm a.s.}{\rightarrow}}
\newcommand{\stj}{\stackrel{J_1}{\rightarrow}}
\newcommand{\stv}{\stackrel{v}{\rightarrow}}
\newcommand{\stw}{\stackrel{w}{\rightarrow}}
\newcommand{\halmos}{\quad\hfill\mbox{$\Box$}}
\newcommand{\DD}{\mathbb{D}}
\newcommand{\RR}{\mathbb{R}}
\newcommand{\MM}{\mathbb{M}}
\newcommand{\NN}{\mathbb{N}}
\newcommand{\ZZ}{\mathbb{Z}}
\newcommand{\CC}{\mathbb{C}}
\newcommand{\FF}{\mathbb{F}}
\newcommand{\PP}{\mathbb{P}}
\newcommand{\QQ}{\mathbb{Q}}
\newcommand{\UU}{\mathbb{U}}
\newcommand{\myPP}{\mathbb{P}}
\newcommand{\mySS}{\mathbb{S}}
\newcommand{\TT}{\mathbb{T}}
\newcommand{\JJ}{\mathbb{J}}
\newcommand{\II}{\mathbb{I}}
\newcommand{\XX}{\mathbb{X}}
\newcommand{\YY}{\mathbb{Y}}
\newcommand{\PPP}{{\cal P}}
\newcommand{\MMM}{{\cal M}}
\newcommand{\AAA}{{\cal A}}
\newcommand{\BBB}{{\cal B}}
\newcommand{\CCC}{{\cal C}}
\newcommand{\DDD}{{\cal D}}
\newcommand{\FFF}{{\cal F}}
\newcommand{\GGG}{{\cal G}}
\newcommand{\JJJ}{{\cal J}}
\newcommand{\LLL}{{\cal L}}
\newcommand{\EEE}{{\cal E}}
\newcommand{\SSS}{{\cal S}}
\newcommand{\TTT}{{\cal T}}
\newcommand{\HHH}{{\cal H}}
\newcommand{\III}{{\cal I}}
\newcommand{\KKK}{{\cal K}}
\newcommand{\NNN}{{\cal N}}
\newcommand{\OOO}{{\cal O}}
\newcommand{\RRR}{{\cal R}}
\newcommand{\UUU}{{\cal U}}
\newcommand{\VVV}{{\cal V}}
\newcommand{\WWW}{{\cal W}}
\newcommand{\ZZZ}{{\cal Z}}
\newcommand{\EEEE}{\mathfrak{E}}
\newcommand{\MMMM}{\mathfrak{M}}
\newcommand{\GGC}{GGC}
\newcommand{\VGGC}{VGG}
\newcommand{\Levy}{L\'evy }
\newcommand{\skal}[2]{\left\langle #1,#2\right\rangle}
\newcommand{\mylin}{\mbox{lin}}
\newcommand{\sign}{\mbox{sign}}
\newcommand{\var}{\mbox{var}}
\newcommand{\id}{\mbox{id}}
\newcommand{\eins}{{\bf 1}}
\newcommand{\eeee}{\mathfrak{e}}
\newcommand{\bfnull}{{\bf 0}}
\newcommand{\bfs}{{\bf s}}
\newcommand{\bft}{{\bf t}}
\newcommand{\bfx}{{\bf x}}
\newcommand{\bfy}{{\bf y}}
\newcommand{\bfd}{{\bf d}}
\newcommand{\bfalpha}{\boldsymbol{\alpha}}
\newcommand{\bfgamma}{\boldsymbol{\gamma}}
\newcommand{\bflambda}{\boldsymbol{\lambda}}
\newcommand{\bfmu}{\boldsymbol{\mu}}
\newcommand{\bfpi}{\boldsymbol{\pi}}
\newcommand{\bftheta}{\boldsymbol{\theta}}
\newcommand{\artanh}{\mbox{artanh}}
\newcommand{\rank}{\mbox{rank}}
\newcommand{\erank}{\mbox{\em rank}}
\newcommand{\ACC}{{\cal AC}_C}
\newcommand{\LLC}{{\cal L}^1_C}
\newcommand{\FIDI}{\stackrel{fdd}{\longrightarrow}}
\newcommand{\SKO}{\stackrel{Sko}{\longrightarrow}}
\newcommand{\vstd}{\stackrel{vague}{\longrightarrow}}
\newcommand{\wlimit}{\stackrel{w}{\longrightarrow}}
\newcommand{\vect}[2]{\left(\begin{array}{c}#1\\#2\end{array}\right)}
\newcommand{\myCov}{{\rm Cov}}
\newcommand{\myVar}{{\rm Var}}
\newcommand{\mymatr}[4]{\left(\begin{array}{cc}#1&#2\\#3&#4\end{array}\right)}
\newcommand{\vectdrei}[3]{\left(\begin{array}{c}#1\\#2\\#3\end{array}\right)}
\newcommand{\vectdreiwide}[3]{\left(\begin{array}{c}#1\\\\#2\\\\#3\end{array}\right)}
\newcommand{\diag}{\mbox{diag}}

\newcommand{\Sdpluswithnorm}{\mySS^d_{\|\cdot\|,+}}
\newcommand{\Sdplus}{\mySS^d_{+}}
\newcommand{\spur}{\mbox{trace}}
\newcommand{\BBsubseteq}[1]{
\mathrel{\rotatebox[origin=c]{#1}{$\subseteq$}}}
\newcommand{\BBarrow}[1]{
\mathrel{\rotatebox[origin=c]{#1}{$\boldsymbol{-\!\!\!-\!\!\!\blacktriangleright}
$}}}
\newcommand{\wt}{\widetilde}

\numberwithin{equation}{section}
\newtheorem{thm}{Theorem}[section]
\newtheorem{lem}{Lemma}[section]
\newtheorem{prop}{Proposition}[section]
\newtheorem{cor}{Corollary}[section]

\newtheorem{theorem}{Theorem}
\newtheorem{lemma}{Lemma}
\newtheorem{remark0}{\sc Remark}
\newenvironment{remark}{\begin{remark0}\em}{\end{remark0}\par}
\newtheorem{corollary}{Corollary}
\newtheorem*{reverse}{Reverse direction}
\newtheorem{definition}{Definition}
\newtheorem{proposition}{Proposition}
\numberwithin{theorem}{section}
\numberwithin{proposition}{section}
\numberwithin{lemma}{section}
\numberwithin{corollary}{section}
\numberwithin{remark0}{section}
\numberwithin{definition}{section}

\def\Mr{M^{(r)}}
\def\Yr{Y^{(r)}}
\def\bYr{\boldsymbol{Y}^{(r)}}
\def\bMr{\boldsymbol M^{(r)}}
\def\bm{\boldsymbol{m}}
\def\bmu{\boldsymbol{\mu}}
\def\Lr{L^{(r)}}
\def\Tr{T^{(r)}}
\def\gammai{\gamma^{(i)}}
\def\Z{\mathbb{Z}}
\def\eqd{\stackrel{d}{=}}
\def\nuhatstand{\hat\nu_{\text{standard},n}}
\def\cond{\,|\,}
\def\Starica{St\u{a}ric\u{a} }
\def\Rootzen{Rootz\'en }
\def\SaS{S$\alpha$S\ }
\def\one{{\bf 1}}
\def\Fbar{{\overline F}}
\def\lst{\buildrel{\rm st}\over{\leqs}}
\def\Leb{\text{{\it Leb}}}
\def\essot{s_T^{(0)} }
\def\supp{\text{supp} }
\def\Aff{\text{Aff} }
\def\bX{\boldsymbol X}
\def\bY{\boldsymbol Y}
\def\bU{\boldsymbol U}
\def\bZ{\boldsymbol Z}
\def\bu{\boldsymbol u}
\def\bone{\boldsymbol 1}
\def\ba{\boldsymbol a}
\def\boldb{\boldsymbol b}
\def\bq{\boldsymbol q}
\def\bs{\boldsymbol s}
\def\bt{\boldsymbol t}
\def\bzero{\boldsymbol 0}
\def\be{\boldsymbol e}
\def\bx{\boldsymbol x}
\def\bz{\boldsymbol z}
\def\boldb{\boldsymbol b}
\def\by{\boldsymbol y}
\def\bj{\boldsymbol j}
\def\bJ{\boldsymbol J}
\def\binfty{\boldsymbol \infty}
\def\bTheta{\boldsymbol \Theta}
\def\balpha{\boldsymbol \alpha}
\def\btheta{\boldsymbol \theta}
\def\E{\mathbb{E}}
\def\R{\mathbb{R}}
\def\C{\mathfrak{C}}
\def\V{\boldsymbol V}

\section{Markov Property of Higher Order Extremal
  Processes with Discrete Indexing}\label{sec:Markov}

\subsection{Indexing}\label{sec:ind}
Our analysis requires that we
keep track of infinite sequences indexed by $r$ where the first
members are being moved further out as $r$ increases. To cope with
this we use the idea of shifted sequences, with first members replaced
by $-\infty$. To see how this works,  
we start with the sequence space $\RR^\NN_{-\infty}:=\{\bx=(x_n):x_n\in\RR_{-\infty}\,,n\in\NN\}$ endowed
with the Borel field associated 
with the product topology.
(We employ the notations
$\NN=\{1,2,3,\dots\}$, $\RR_{-\infty}:=\RR\cup\{-\infty\}=[-\infty,\infty)$, and conventions $\sum_\emptyset=0$, $\prod_\emptyset:=1$,
$\pm\infty \times 0=0$.
Also $\RR^{\NN,\uparrow}_{-\infty}=\{\bx=(x_n)\in\RR^\NN_{-\infty}:\,x_{n}\le x_{n+1}\,,n\in\NN\}$
denotes the subset of nondecreasing sequences.)
The partial maxima operator $\bigvee:\RR^\NN_{-\infty}\mapsto \RR^{\NN,\uparrow}_{-\infty}$ maps a given sequence $\bx=(x_n)_n\in \RR^\NN_{-\infty}$ to its associated
sequence of partial maxima $\bigvee
\bx:=(\max\{x_1,\dots,x_n\})_n$. \sid{(In the statistical language R,
  this is known as {\it cummax\/}.)}

For a given sequence $\bx\in\RR_{-\infty}^\NN$ and $r\in\NN$, $n\geq r$, let $m_n^{(r)}$ be the $r$th largest of $x_1,\dots,x_n$, arranged in lexicographical order in case of ties.
 Then set 
$$x_n^{(r)}
=\begin{cases}
-\infty,& \text{ if }n<r;\\
m_n^{(r)},& \text{ if } n\geq r.
\end{cases}
$$
The {\em  extremal sequence of order $r$  associated with $\bx$}
 is the sequence $\bx^{(r)}\in\RR^{\NN,\uparrow}_{\infty}$, with finite elements $x_n^{(r)}$ augmented with $-\infty$ as follows: \begin{equation}\label{augseq}
 \bx^{(r)}=\bigl(\underbrace{-\infty,\dots,-\infty,}_{r-1 \text{
    entries}}m_n^{(r)},n\geq r\bigr).
\end{equation}
Write $\bx^{(0)}:=\bx$ for the extremal sequence of zero order. 
The extremal sequence
of unit order equals the partial maximum sequence: $\bx^{(1)}=\bigvee \bx$.

For a sequence $\bx=(x_n)_{n}\in\RR_{-\infty}^\NN$ the {\it shifted sequence} $\bx_{\mathcal{R}}$ is $  \bx_{\mathcal{R}}=(-\infty,\bx) \in \RR^\NN_{-\infty}$.
For two sequences $\bx=(x_n)_{n}$, $\by=(y_n)_{n}\in\RR_{-\infty}^\NN$, let
\[
\bx {}_{\mathcal{R}}\!\wedge\by:=
\{(-\infty)\eins_{n=1}+(x_{n-1}\wedge
y_n) \eins_{n>1}\}_n\in\RR_{-\infty}^\NN
\]
 be the componentwise minimum
of $\bx$ and $\by$, 
taken after shifting $\bx$ to the right with proper augmentation with
$-\infty$. Thus, componentwise, when $\bfx=(x_1,x_2,\ldots)$ and
 $\bfy=(y_1,y_2,\ldots)$, we have
\begin{equation*}
\bx_{\mathcal{R}}=(-\infty, x_1,x_2, \ldots)
\quad {\rm and}\quad
\bx {}_{\mathcal{R}}\!\wedge\by=(-\infty, x_1\wedge y_2, x_2\wedge y_3, \ldots).
\end{equation*}
For $n\in\NN$,  $y_n^{(1)}\ge y^{(2)}_n\ge \dots\ge y^{(n)}_n$ denotes
the order statistics  associated with (possibly extended)  real
numbers $y_1,\dots,y_n\in\RR_{-\infty}$. \sid{Clearly}, this notation is
consistent with the previous.

 In
Theorem \ref{thmswitchr},
we will
show a Markov property for the $r$th largest of an 
\sid{ iid sequence, and
since recursions are an effective tool for proving a sequence of random
elements is
Markovian, we first prove a preliminary result focussing on properties of the shifted sequences.}  

\begin{prop}\label{prop:rec}
For $r\in\NN$, we have the identity,
\begin{equation}\label{e:rec}
\bx^{(r+1)} =\bigvee (\bx^{(r)} {}_{\mathcal{R}}\!\wedge \bx)
\end{equation} 
or in component form,
\begin{equation}\label{e:recA}
x_n^{(r+1)}=\bigvee_{j=r+1}^n \Bigl(x_{j-1}^{(r)}\wedge x_j\Bigr),\quad r\in
\NN,\,n\geq r+1.
\end{equation}
\end{prop}

\begin{proof} Fix an integer $r$ and we prove \eqref{e:recA} by
  induction on $n$. The base of the induction is $n=r+1$ and the left
  side of \eqref{e:recA} is $x_{r+1}^{(r+1)}=\wedge_{i=1}^{r+1}
  x_i$. The right side is $x_r^{(r)}\wedge x_{r+1}=\wedge_{i=1}^{r+1}x_i$. 
  So  \eqref{e:recA} is proved for   $n=r+1$.

As an induction hypothesis, assume  \eqref{e:recA}  is true for
$n=r+p$ for $p\geq 1$ and we verify \eqref{e:recA}  is true for
$n=r+p+1$.
The left side of \eqref{e:recA}  for $n=r+p+1$ is
$x_{r+p+1}^{(r+1)}=LHS.$ The right side is
\begin{align}
RHS=& \bigvee_{j=r+1}^{r+p+1} \bigl(x_{j-1}^{(r)}\wedge x_j\bigr)
=\bigvee_{j=r+1}^{r+p} \bigl(x_{j-1}^{(r)}\wedge x_j\bigr) \bigvee 
 \bigl(x_{r+p}^{(r)}\wedge x_{r+p+1}\bigr)\nonumber\\
\intertext{and from the induction hypothesis this is equal to}
&x_{r+p}^{(r+1)} \bigvee  \bigl(x_{r+p}^{(r)}\wedge x_{r+p+1}\bigr).\label{e:refer}
\end{align}
Now consider cases:
\begin{description}
\item[Case (a) $x_{r+p+1}>x_{r+p}^{(r)}$.] For this case, increasing
  the sample size from $r+p$ to $r+p+1$ means $x_{r+p}^{(r)}$  becomes
  $x_{r+p+1}^{(r+1)}$. So $RHS= x_{r+p}^{(r+1)} \bigvee
  x_{r+p}^{(r)}=x_{r+p}^{(r)}=LHS$.
\item[Case (b) $x_{r+p}^{(r+1)} \leq x_{r+p+1} \leq x_{r+p}^{(r)}$.]
  The term in parentheses on the right side of \eqref{e:refer} is 
$$x_{r+p}^{(r)} \wedge x_{r+p+1}= x_{r+p+1} = x_{r+p+1}^{(r+1)}$$
and thus 
$$RHS=x_{r+p}^{(r+1)} \vee x_{r+p+1}^{(r+1)} = x_{r+p+1}^{(r+1)}  =LHS.$$
\item[Case (c) $x_{r+p+1}< x_{r+p}^{(r+1)}$.]  We have
$$RHS=x_{r+p}^{(r+1)} \vee \bigl(x_{r+p}^{(r)} \wedge x_{r+p+1}\bigr)
=x_{r+p}^{(r+1)} \vee x_{r+p+1}=x_{r+p}^{(r+1)} $$
and because of where the added point $x_{r+p+1}$ is located, when the
sample size increases from $r+p$ to $r+p+1$, the above
equals $x_{r+p+1}^{(r+1)} =LHS$.
\end{description}
The three cases exhaust the possibilities and this completes the
induction argument.
\end{proof}

\subsection{The IID Setting}\label{sec:iid}
Now we add the randomness. Let $\bX=(X_n)_n\in\RR^\NN$ be an iid sequence of rvs in $\RR$ with cdf $F$ and set $ \bX^{(0)}= \bX$.
Then for $r\in\NN$ the {\it $r$-th order extremal process}
 is  the augmented sequence 
 $\bX^{(r)}=(X^{(r)}_n)_{n\in\NN}$ in $\RR^\NN_{-\infty}$ constructed as in \eqref{augseq}; specifically,
\begin{equation}\label{augseqR}
 \bX^{(r)}=\bigl(\underbrace{-\infty,\dots,-\infty,}_{r-1 \text{
    entries}}M_n^{(r)},n\geq r\bigr),
\end{equation}
where the $M_n^{(r)}$ are the order statistics of $X_1,X_2,\ldots, X_n$ 
defined lexicographically as for the $m_n^{(r)}$ in \eqref{augseq}.
Note that $\bX^{(1)}=\bigvee \bX^{(0)}=\bigvee \bX$ is the sequence of partial maxima associated with $\bX$.

\sid{
To think about the Markov property for $(\bX^{(r)},r\geq 1)$,
we imagine conditioning on the monotone sequence $\bX^{(r)}
=\bx^{(r)}$. For indices where  $\bx^{(r)}$ is constant, say $x$, 
 the structure of  $\bX^{(r+1)}$ should be as if we
construct the maximum sequence from repeated observations from the
conditional distribution of $(X_1|X_1 \leq x)$. The following
construction  make this precise.}

Let $\bU=(U_{r,n})_{n,r\in\NN}$ be an iid array of uniform r.v.'s in $(0,1)$.
Assume $\bX=\bX^{(0)}$ and $\bU$ are independent random elements.
For $m\in\RR$ with $F(m)>0$
the left-continuous inverse $u\mapsto F^\leftarrow(u|m)$ of the
conditional cdf $x\mapsto F(x|m):=P(X_{\sid{1}}\le x|X_{\sid{1}}\le m)$ is well-defined;
otherwise, if $F(m)=0$ set $F^\leftarrow(u|m)=\eins_{m> 0}$ with
$F^\leftarrow(u|-\infty)\equiv 0$.

For $r\in\NN=\{1,2,\dots\}$ introduce two sequences 
$\widehat \bX_{(r+1)}=(\widehat X_{(r+1),n})_n$ and 
$\wt \bX_{(r+1)}=(\wt  X_{(r+1),n})_n$.
For the first, we have for $n=1$ that $\widehat
X_{(r+1),1}:=X_1^{(1)}=X_1$ and, for $n\ge 2$,
$$
\widehat X_{(r+1),n}:=\begin{cases}
F^\leftarrow(U_{r,n}|X_n^{(r)})\prod_{1\le k\le r}
\eins_{X^{(k)}_n=X^{(k)}_{n-1}},& \text{ if }
X_{n-1}^{(r)}=X_n^{(r)}\\
\sum_{k=1}^r X^{(k)}_n\eins_{X^{(k)}_n>X^{(k)}_{n-1}}\prod_{1\le l<k}
\eins_{X^{(l)}_n=X^{(l)}_{n-1}},& \text{ if } X_{n-1}^{(r)}<X_n^{(r)}\\
\end{cases}
$$
so if there is no jump in the $r$th order maximum process we sample
from the conditional distribution and if there is a jump, we note the
new value that caused the jump. For the second sequence we have
$\wt X_{(r+1),n}:=-\infty $ if $n\leq r$ and if $n>r$ 
$$
\wt X_{(r+1),n}:=\begin{cases}
X_{n-1}^{(r)}
,& \text{ if }X_n^{(r)}>X^{(r)}_{n-1},\\
F^\leftarrow(U_{r,n}|X^{(r)}_n)
,&\text{
    if }X^{(r)}_n=X^{(r)}_{n-1},\\
\end{cases}
$$
so if there is no jump in the $r$th order maxima at $n$, we
sample from the conditional distribution and if there is a jump at
index $n$ we note the smaller value at $n-1$ that the process jumps from.
The sequences $\wt \bX_{r+1}$ and $\widehat \bX_{(r+1)}$ depend on
$\bX,\bX^{(1)},\dots,\bX^{(r)}$ only via $\bX^{(r)}$ and
$\bX^{(1)},\bX^{(2)},\dots,\bX^{(r)}$,  respectively. 

\subsection{Identities in Law and the Markov Property}\label{sec:idL}
Next we provide some identities in law which will show that the sequence of extremal processes is a sequence-valued Markov chain.

\begin{theorem}\label{thmswitchr}
For $r\in\NN$ the following random variables are equal in distribution
as random elements in $(\RR_{-\infty}^\NN)^{(r+1)}$ \sid{(and hence in 
$(\RR_{-\infty}^\NN)^{\NN}$),}
\begin{equation}\label{XisperpetMAX}(\bX^{(0)},\dots,\bX^{(r)})\;\eqd\;(\widehat
  \bX_{(r+1)},\bX^{(1)},\dots,\bX^{(r)})\,,\end{equation} 
and
\begin{equation}\label{choppedr}\Big(\bX^{(1)},\dots,\bX^{(r+1)}\Big)\;\eqd\;\Big(\bX^{(1)},\dots,\bX^{(r)},\bigvee \wt \bX_{(r+1)}\Big)\,.\end{equation}
In particular,~$\bX^{(1)},\bX^{(2)}\dots$ is a Markov chain with state space $\RR^{\NN,\uparrow}_{-\infty}$, with its conditional distributions satisfying
\begin{equation}\label{MARKOVr}\Big(\bX^{(r+1)}\Big|\bX^{(r)},\dots,\bX^{(1)}\Big)\;\eqd\;\Big(\bigvee \wt \bX_{(r+1)}\Big|\bX^{(r)}\Big)\,,\quad r\in\NN\,.\end{equation}

\end{theorem}

\begin{proof} Indeed,~\eqref{choppedr} follows from~\eqref{XisperpetMAX} because
\begin{align*}
\big(\bX^{(1)},\dots,\bX^{(r+1)}\big)=&\big(\bX^{(1)},\dots,\bigvee
                                        (\bX^{(r)}{}_{\mathcal{R}}\!\wedge
                                        \bX^{(0)})\big)&&(\text{Proposition
  \ref{prop:rec}}),\\ 
\eqd&\big(\bX^{(1)},\dots,\bX^{(r)},\bigvee (\bX^{(r)}{}_{\mathcal{R}}\!\wedge \widehat \bX_{(r\!+\!1)})\big)&&\text{(from }\eqref{XisperpetMAX})\\ 
=&\big(\bX^{(1)},\dots,\bX^{(r)},\bigvee  \wt
   \bX_{(r+1)}\big)&&\text{(definitions)}\,.
\end{align*}
In~\eqref{choppedr} $\wt \bX_{r+1}$ depends on $\bX^{(1)},\dots,\bX^{(r)}$ only through $\bX^{(r)}$, and this holds for all $r\in\NN$. In particular,~\eqref{MARKOVr} holds, and $\bX^{(1)},\bX^{(2)},\dots$ must be a Markov chain.

It remains to show~\eqref{XisperpetMAX}. For $r\in\NN$ let
$\RR_{-\infty}^{r,\downarrow}:=\{\bm=(m_1,\dots,m_r)\in\RR_{-\infty}^r:m_1\ge{\dots}\ge
m_r\}$ be the space of $r$-tuples with nonincreasing $\RR_{-\infty}$-valued
components, and introduce a smooth truncation mapping
$\bmu_r=(\mu_{r,1},\dots,\mu_{r,r}):\RR_{-\infty}^{r,\downarrow}\times
\RR\mapsto \RR_{-\infty}^{r,\downarrow}$, by setting $\mu_{r,1}(m,x):=
x\vee m_1$, and, for $2\le k\le r$,
\[\mu_{r,k}(\bm,x)=m_{k-1}\eins_{x>m_{k\!-\!1}}
+m_{k}\eins_{x\le m_{k}}+ x\eins_{m_{k}<x\le m_{k-1}}\,,\]
when $x\in\RR$ and $\bm=(m_1,\dots,m_r)\in\RR_{-\infty}^{r,\downarrow}$.
Note that
\begin{equation}\label{e:mu>}
\mu_{r,k}(\sid{\bm,x)}\ge m_k \text{ for }
\bm\in\RR_{-\infty}^{r,\downarrow}, \,x\in\RR,\,1\le k\le r.
\end{equation}
 Also,
define mappings 
$\widetilde
\bmu_r=
(\widetilde\mu_{r,0},\dots,\widetilde\mu_{r,r}):\RR_{-\infty}^{r,\downarrow}
\times~\RR\sid{\mapsto} \RR\times\RR_{-\infty}^{r,\downarrow}$ and
$\widehat\bmu_r=(\widehat\mu_{r,0},\dots,\widehat\mu_{r,r}):\RR_{-\infty}^{r,\downarrow}\times
\RR\times(0,1)\mapsto \RR\times\RR_{-\infty}^{r,\downarrow}$, by
setting 
$$\widetilde
\mu_{r,k}(\bm,x):=\widehat\mu_{r,k}(\bm,x,u):=\mu_{r,k}(\bm,x),\quad
1\le k\le r, $$ 
and with $k=0$, $\widetilde\mu_{r,0}(\bm,x):=x$ and 
\begin{align}
\widehat\mu_{r,0}&(\bm,x,u)=
F^\leftarrow(u|\mu_{r,r}(\bm,x))
\prod_{1\le k\le r} \eins_{\mu_{r,k}(\bm,x)=m_k}
\nonumber 
\\
&{} +\sum_{k=1}^r
  \mu_{r,k}(\bm,x)\eins_{\mu_{r,k}(\bm,x)>m_k}\prod_{1\le l<k}
  \eins_{\mu_{r,l}(\bm,x)=m_l}\nonumber \\ 
=&F^\leftarrow(u|m_r)
\prod_{1\le k\le r} \eins_{\mu_{r,k}(\bm,x)=m_k} +x
   \sum_{k=1}^r\eins_{\mu_{r,k}(\bm,x)>m_k}\prod_{1\le l<k}
   \eins_{\mu_{r,l}(\bm,x)=m_l}\nonumber \\
=&\begin{cases}
F^\leftarrow(u|m_r),& \text{ if }x\leq m_r,\\
x,& \text{ if }x>m_{r}.\end{cases}
\label{e:hatmu0}\end{align}
for $\bm=(m_1,\dots,m_r)\in\RR_{-\infty}^{r,\downarrow}$, $x\in\RR$ and $u\in(0,1)$.

One can check that the component form of the left and right sides of \eqref{XisperpetMAX} is for $n\geq 2$,
\begin{align}
\Bigl((X_n,X_{n}^{(1)},\dots,X^{(r)}_{n}), n\geq 2\Bigr)=&
\Bigl(
\widetilde\bmu_r(X_{n-1}^{(1)},\dots,X^{(r)}_{n-1},X_n), n\geq 2
\Bigr) \label{e:1}\\
\Bigl(( \hat X_{(r+1),n},X_{n}^{(1)},\dots,X^{(r)}_{n}),n\geq 2\Bigr)=&
\Bigl(
\widehat\bmu_r(X_{n-1}^{(1)},\dots,X^{(r)}_{n-1},X_n,U_{r,n}), n \geq 2\Bigr),\label{e:2}
\end{align}
where $X_n \independent (X_{n-1}^{(1)},\dots,X^{(r)}_{n-1})$ and
$U_{r,n}\independent (X_{n-1}^{(1)},\dots,X^{(r)}_{n-1},X_n,)$ since
we assumed that $\bX$ and $\bU$ are independent arrays of iid
rv's. The right sides of \eqref{e:1} and \eqref{e:2} are 
Markov chains with stationary transition probabilities in the index
$n$ (new value is a function of the 
previous value and an independent quantity) and for $n=1$, the left
sides of \eqref{e:1} and \eqref{e:2} have common initial value 
$(X_1,X_1,-\infty,\dots,-\infty) \in \mathbb{R}\times
\mathbb{R}_{-\infty}^{r,\downarrow}$. Therefore, to prove equality in 
distribution in \eqref{XisperpetMAX},
it suffices to prove both chains have a common transition kernel.

To see this, 
let $X'\eqd X_1 \sim F$ and $U'\eqd U_{1,1}\in(0,1)$ be independent
rv's. For $x,y\in\RR$ with $F(y)>0$ note 
\begin{align}
P(X'\!\le\! y,F^\leftarrow(U'|y)\le x)=&P(X'\!\le\! y)P(X'\!\le\!
                                         x|X'\!\le\! y)\nonumber \\
=&F(y)F(x|y)=F(x\wedge y),\label{e:extra}
\end{align}
Consequently, for
$\bm=(m_1\!,\!\dots\!,\!m_r),\,\bm'=(m'_1\!,\!\dots\!,\!m'_r)\in\RR_{-\infty}^{r,\downarrow}$
with $F(m'_k)\!>\!0$ for $1\!\le\! k\!\le\! r$, setting
$m'_{0}:=\infty$, we have for the transition probability,
\begin{align*}
P\Bigl(
\bigl( &\widehat X_{(r+1),n+1}, X_{n+1}^{(1)},\dots ,X_{n+1}^{(r)} \bigr)
\in (-\infty,x]\times\prod_{k=1}^r [-\infty, m_k]\\
&\qquad \qquad \Big|
\widehat X_{(r+1),n}=y, (X_{n}^{(1)},\dots ,X_{n}^{(r)}) =\bm'\Bigr)\\
&=P\Big(\widehat\bmu_r(\bm',X',U')\in (-\infty,x]\times\prod_{k=1}^r
  [-\infty, m_k]\Big)\\ 
&=P\Big(\widehat\bmu_r(\bm',X',U')\in (-\infty,x]\times\prod_{k=1}^r
  [-\infty, m_k], X'\leq  m_r'\Big)\\ 
&\qquad +\sum_{k=1}^r P\Big(\widehat\bmu_r(\bm',X',U')\in
  (-\infty,x]\times\prod_{k=1}^r 
  [-\infty, m_k], X' \in (m'_k, m_{k-1}']\Bigr)\\&=A+B.
\end{align*}
Consider 
$$A=P\Bigl( F^\leftarrow (U'|m_r')\leq x, \mu_{rk}(\bm',X',U')\leq
m_k, k=1,\dots,r; X'\leq m_r'\Bigr).$$
 If $m_k'>m_k$ for some $k=1,\dots,r$, then because of \eqref{e:mu>},
 the probability A is 0. So assume for $k=1,\dots,r,$ that $m_k'\leq
 m_k$. Then the condition $X'\leq m_r'$ in $A$ implies $X'\leq
 m_k'\leq m_k$ for $k=1,\dots,r$ and using \eqref{e:extra}, $A$ reduces to
$$A=P(F^\leftarrow (U'|m_r')\leq x, X'\leq m_r') =F(x\wedge m_r')\prod_{1\le k\le r} \eins_{m_k'\le m_k} .
$$
For $B$ we use \eqref{e:hatmu0} and get
\begin{align*}
B&=\sum_{k=1}^r P(X'\in (m_k',m_{k-1}'], X'\leq x,
\underbrace{\mu_{rl}(\bm',X')}_{\geq m_l'}\leq m_l;l=1,\dots,r)
\end{align*}
Fix $k$ and suppose $l>k$. Then the interval $(m_l',m_{l-1}']$ is to
the left of $(m_k',m_{k-1}']$  where $X'$ is located and
$\mu_{rl}(\bm',X')=m_{l-1}'$. The probability is then $0$ unless
  $m_l\geq m_{l-1}'$. If $l<k$, the order of the intervals is reversed,  $\mu_{rl}(\bm',X')=m_l'$, and the probability is $0$ unless
  $m_l'\leq m_l$. Thus, $B$ becomes
\begin{align*}
B&=\sum_{k=1}^r P(m_k'<X'\le x\wedge m_k\wedge m_{k-1}') \prod_{1\le
   l<k}\eins_{m_l'\le m_l}\;\prod_{k<l\le r}\eins_{m_{l-1}'\le m_l}. 
\end{align*}

On the other hand, from the left sides of \eqref{XisperpetMAX} and
\eqref{e:1},
\begin{align*}
P\Bigl ( \bigl(X_{n+1},&X_{n+1}^{(1)},\dots,X_{n+1}^{(r)}\bigr) \in 
(-\infty,x]\times\prod_{k=1}^r [-\infty, m_k]\\
&\qquad \qquad\Big|X_n=y,
(X_{n},X_{n}^{(1)},\dots,X_{n}^{(r)})=\bm')\\
&=P(\widetilde{\bmu}_r (\bm',X') \in (-\infty,x]\times\prod_{k=1}^r
  [-\infty, m_k])\\
&=P\Bigl( \bigl( X',\mu_{rl}(\bm',X'), l=1,\dots,r\bigr) \in (-\infty,x]\times\prod_{k=1}^r
  [-\infty, m_k] \Bigr)\\
&=P\Bigl( X' \leq x, X'\leq m_r',\mu_{rl}(\bm',X') \leq m_l,
  l=1,\dots,r)\\
&\qquad +\sum_{k=1}^r P( X'\leq x, X'\in (m_k,m_{k-1}],\mu_{rl}(\bm',X') \leq m_l,
  l=1,\dots,r)\\
&=A+B.
\end{align*}
This completes the proof of~\eqref{XisperpetMAX} and of Theorem \ref{thmswitchr}.
\end{proof}

\section{Asymptotic Behaviour of the Discrete Time Process $\bMr$ for large $r$}\label{sec:asymp}
In this section we consider asymptotic behaviour as 
 $r\to\infty$  of $\{\bMr=(\Mr_n, n \geq r), r\geq 1\}$ as an $\R^\infty$-valued stochastic process.   
 As $r$ increases we are pushing into values far from the largest, so
 limit behaviour for both
the range of $\bMr$ and $\bMr$ itself, depend critically on left tail
behavior of the distribution of $X_1$.
Appropriate left-tail conditions related to  domain
 of attraction conditions in classical extreme value theory 
 make the range and the sequence
 of $r$th order maxima converge weakly.



Throughout this section we will assume  $F$ is
continuous, so the records of $\{X_n\}$ are Poisson with mean measure
$R$ \citep[page 166]{resnickbook:2008} which we denote PRM($R$). 
The assumption of continuity could be relaxed as in
\cite{engelen:tommassen:vervaat:1988,shorrock:1974,shorrock:1975} but
results are most striking when $F$ is continuous and we proceed in
this setting.
\subsection{$r$th maximum and $r$-records}
Let $\{X_n, n\ge 1\} $ be iid random variables with common distribution
function $F(x)$ and 
 set $R(x)=-\log (1-F(x))=-\log \bar F(x).$ 
Define
\begin{align*}
R_n=&\sum_{j=1}^n 1_{[X_j \geq X_n]}
=\text{relative rank of $X_n$ among $X_1,\dots,X_n$}\\
=&\text{rank of $X_n$ at ``birth''}.
\end{align*}
It is known \citep{renyi:1962} that $\{R_n\}$ are independent random variables and $R_n$
is uniformly distributed on $\{1,\dots,n\}$; that is,
$$P[R_n=i]= 1/n,\quad i=1,\dots,n.$$

Considering $\{\bMr, r\geq 1\}$ as an $\R^\infty$-valued stochastic
process, we ask what is the asymptotic behavior  of $\bMr$ and its range as a function of $r$  as $r\to\infty$?

Define the $r$-record times of $\{X_n\}$ by
$$\Lr_0 =0,\quad \Lr_{n+1}=\inf\{j>\Lr_n: R_j=r\}.
$$
The $r$-records are then $\{X_{\Lr_n}, n\geq 1\}$, which for each $r$,
are points of
PRM($R(dx)$) by Ignatov's theorem.


We list some initial facts about $\bMr$ and its range.
\begin{itemize}
\item For fixed $r$,  $\bMr=\{M_n^{(r)},n\geq r\}$ jumps at index $k
  \geq r$ iff 
$$R_k \in \{1,\dots,r\},$$
so 
$$\{[\bMr \text{ jumps at index }k],k\geq r\}$$
are independent events over $k$ and
$$P[\bMr \text{ jumps at }k] =\frac rk.$$

\begin{Remark}
{\rm This has the implication that if we re-index and set $k=r+l$ for
$l\geq 0,$ then for any fixed $l$,
$$P[\bMr \text{ jumps at }r+l] =\frac{ r}{r+l} \to 1,\quad
(r\to\infty).$$
So for large $r$, $\bMr$ jumps at almost every integer. Define the
jump indices
$$\{\tau_l^{(r)},l\geq 0\} =\{j\geq 1: M_{r+j}^{(r)}>M^{(r)}_{r+j-1}\} \cup
\{0\}.$$
Then in $\mathbb{R}_+^\infty$,
$$ \{\tau_l^{(r)},l\geq 0\} \Rightarrow \{0,1,2,\dots \}.$$
}
\end{Remark}

\item For fixed $r$, let $\mathcal{R}_r$ be the range of $\bMr$;
that is, the distinct points without repetition hit by $\{M_n^{(r)},
n\geq r\}$. Then,
\beqq\label{e:rangeM}
\mathcal{R}_r:=\bigcup_{p=1}^r \{X_{L_n^{(p)}}, n\geq 1\},\eeqq
By Ignatov's theorem \citep{ignatov:1977,
  stam:1985,goldie:rogers:1984,engelen:tommassen:vervaat:1988,resnickbook:2008},
this is a sum of $r$ independent PRM(R) 
processes and therefore the range of $\bMr$ is PRM($rR$).

To prove \eqref{e:rangeM}, suppose $\Mr_n=x$, for some $n\geq r.$ Suppose the $r$th
largest of $X_1,\dots,X_n$ occurs at $X_i=x$ for $i \leq n$. If the
rank of $X_i$ were $>r$, it could not be the case that $\Mr_n=x$. This
shows that
$$ \text{range of $\bMr$} \subset 
\bigcup_{p=1}^r \{X_{L_n^{(p)}}, n\geq 1\}.$$
Conversely, suppose $X_{L_n^{(p)}} =x$, so at time $L_n^{(p)}$, the rank
of $X_{L_n^{(p)}}$ is $p$. Wait until $r-p$ additional $X$'s have been
observed that exceed $x$ and then the $r$th largest will equal $x$.
\end{itemize}

\subsection{Limits for the range $\mathcal{R}_r$ of
  $\bMr$}\label{sec:limRange} 
Although our primary interest is in the behavior of $\{\bMr, r\geq
1\}$ as an $\R^\infty$-valued random sequence, it is instructive and
helpful to discuss the behavior of the range $\mathcal{R}_r$ of $\bMr$.

As a basic result we derive a deterministic limit for $\mathcal{R}_r$. 
Let $\mathcal{R}$ be the support of the measure
  $R(\cdot)$ which corresponds to the monotone function
  $R(x)=-\log(1-F(x))$. 

\bpr \label{prop:detLim}
As $r\to\infty$,
$\mathcal{R}_r$, the range of $\bMr$, converges as a random
  closed set in the Fell topology
  \citep{molchanov:2005,matheron:1975,vervaat:holwerda:1997} to the
  non-random limit $\mathcal{R}$: 
\begin{equation}\label{e:suppConv}
\mathcal{R}_r \Rightarrow \mathcal{R}.\end{equation}
\epr

\begin{proof}
Since $\mathcal{R}_r \subset \mathcal{R}$, it
suffices to show for any open $G$ with $\mathcal{R}\cap G \neq
\emptyset$, that
$$P[\mathcal{R}_r \cap G \neq \emptyset] \to 1.$$
However, $\mathcal{R}\cap G \neq
\emptyset$ implies $R(G)>0$ and therefore,
\begin{align*}
P[\mathcal{R}_r \cap G \neq \emptyset] =&
1-P[ \text{PRM}(rR(G))=0]\\
=&1-e^{-rR(G)} \to 1,\quad (r\to\infty)\end{align*}
since $R(G)>0$.
\end{proof}

The set convergence in \eqref{e:suppConv} is to a deterministic
limit. Since $\mathcal{R}_r$ is a PRM($rR$) point process, we can get
a random limit if we center and scale the $\{X_n\} $
so that the
mean measure $rR$ converges to a Radon measure.  Recall $R(x)=-\log \bar
F(x)$.

Assume there exist $a_r>0$ and $b_r \in \mathbb{R}$ and a
non-decreasing limit
function $g(x)$ with more than one point of increase  such that
\begin{equation}\label{e:assume}
rR(a_rx-b_r) \to g(x),\qquad (r\to\infty).
\end{equation}
For $x$ such that $g(x)>0$, to counteract $r\to\infty$, we must have
$R(a_rx-b_r)\to 0$ and 
 $a_rx-b_r$ converging to the
left endpoint of $F$ (and $R$).

We now explain why $e^{-g}$ is related to an extreme value
distribution.
Remembering that $e^{-R}=\bar F$, equation \eqref{e:assume} is equivalent to 
$$(\bar F (a_r x -b_r))^r =\exp\{-rR(a_rx-b_r)\} \to e^{-g(x)}$$
or 
\begin{equation}\label{e:min}
P\Bigl[ \frac{\wedge_{i=1}^r X_i +b_r}{a_r} >x \Bigr] \to
e^{-g(x)}.\end{equation}
So we recognize $e^{-g}$ as the survivor function of an extreme value
distribution of minima of iid random variables. Expressing this in
terms of maxima by setting $Y_i=-X_i$ we get
\eqref{e:min} equivalent to 
\begin{equation}\label{e:max}
P\Bigl[ \frac{\vee_{i=1}^r Y_i -b_r}{a_r} \leq -x \Bigr] \to
e^{-g(x)}   =G_\gamma (-x),\end{equation}
for some $\gamma \in \mathbb{R}$, where $G_\gamma (x)=\exp\{-(1+\gamma
x)^{-1/\gamma}\},\,1+\gamma x>0$ is the shape
parameter family of extreme value 
distributions for maxima \citep{resnickbook:2008,
  dehaan:ferreira:2006}.
So in \eqref{e:assume}, $g(x)=g_\gamma(x)=-\log G_\gamma (-x).$ The
usual way to write \eqref{e:max} is
$$rP[Y_1>a_r(-x)+b_r]\to g(x), \qquad \forall x
\;\text{s.t. }g(x)>0,$$ and \eqref{e:assume} is the same as 
\begin{equation}\label{e:assumeF}
rF(a_rx -b_r)\to g(x), \qquad \forall x
\;\text{s.t. }g(x)>0.
\end{equation}

In particular, apart from centering, we have the cases:
\begin{enumerate}
\item {\it Gumbel case}: $\gamma=0$. Then 
$$g_0(x)=e^x,\quad x\in \mathbb{R}.$$

\item {\it Reverse Weibull case}: $\gamma<0$: Then $1+\gamma (-x)>0$  iff $x>-1/|\gamma|$ and 
$$g_\gamma(x)=(1+|\gamma |x)^{1/|\gamma|},\quad x>-1/|\gamma |.$$
Adjusting the centering and scaling by taking $b_r=0$, we find $R$ is regularly
varying at $0$ and 
$$rR(a_r x)\to x^{1/|\gamma |}, \quad x>0.$$

\item {\it Frech\'et case}: $\gamma>0$.  Then $1+\gamma (-x)>0$  iff $x<1/\gamma$ and 
$$g_\gamma(x)=(1-\gamma x)^{-1/\gamma},\quad x<1/\gamma.$$
Adjusting the centering and scaling so the support is $(-\infty,0)$ we
get 
$$rR(a_rx)\to |x|^{-1/\gamma},\quad x<0,$$
which is regular variation at 0 from the left.
\end{enumerate}

We can apply this analysis to get convergence of $\mathcal{R}_r$ after centering and scaling. 
Recall $\mathcal{R}_r$ is PRM(rR). A family of Poisson point measures
converges weakly iff the mean measures converge
(eg. \cite{resnickbook:2007}). So replacing 
$$X_i\mapsto \frac{ X_i +b_r}{a_r}$$
rescales the points of the range to be Poisson with mean measure given
by the left side of  \eqref{e:assume}.  Let 
\begin{equation}\label{e:defsuppgamma}
\text{supp}_\gamma=\{x:1-\gamma x>0\}
\end{equation}
 and 
$m_\gamma (\cdot) $ be the measure with density
$g'_\gamma (x),\, x\in \text{supp}_\gamma.$
 Let  $M_+(\text{supp}_\gamma)$ be
the space of Radon measures on $\text{supp}_\gamma$, topologized by vague
convergence.
Then \eqref{e:assume} implies the vague convergence
$$rR\bigl(a_r (\cdot) -b_r \bigr)\stackrel{v}{\to} m_\gamma (\cdot) $$
in $M_+(\text{supp}_\gamma)$, and thus on  $M_+(\text{supp}_\gamma)$ we have
\begin{equation}\label{e:rangeGo}
(\mathcal{R}_r +b_r)/a_r \Rightarrow PRM(m_\gamma).
\end{equation}

We may realize $\text{PRM}(m_\gamma )$ as follows: 
Let  $\Gamma_i=\sum_{j=1}^i E_j$ be a sum of iid standard exponential
random variables. The $\{\Gamma_i\}$ are points of a homogeneous
Poisson process rate $1$ on $[0,\infty)$. The measure $m_\gamma$ has
distribution 
$$g_\gamma: \text{supp}_\gamma \mapsto (0,\infty),$$ with inverse
$$g^\leftarrow_\gamma:  (0,\infty) \mapsto \text{supp}_\gamma .$$ 
The transformation theory for Poisson processes (eg. \cite[Section
5.1]{resnickbook:2007}) means
$\sum_{i=1}^\infty \epsilon_{g_\gamma^\leftarrow (\Gamma_i)} 
$ is PRM($m_\gamma$) on $\text{supp}_\gamma$. For instance, if
$\gamma=0$,
$\text{supp}_0=\mathbb{R}$, $g_0(x)=e^x,\,x\in \mathbb{R}, $ and
$g_0^\leftarrow (y)=\log y,\,y>0, $ and PRM$(m_0)=\sum_i
\epsilon_{\log \Gamma_i}.$ 

\subsection{Weak convergence of the $r$th maxima sequence $\bMr$}
\label{sec:rMax}
Having understood how to get the range $\mathcal{R}_r$  of $\bMr$ to
converge, we turn to convergence of $\bMr$ itself.
We continue to suppose the minimum domain of attraction condition, so that
 $R$ satisfies \eqref{e:assume}, and recall  $M_+(\text{supp}_\gamma)$ is
the space of Radon measures on $\text{supp}_\gamma$, topologized by vague
convergence. 
We start with a preliminary result on the empirical measures
generated by $\{X_i\}$ that will be needed to study the weak convergence of $\{\bMr\}$.

\begin{Proposition}\label{prop:ptproc}
Assume \eqref{e:assume}. If
$N$ is a random element  of $M_+(\text{supp}_\gamma)$ which is PRM$(m_\gamma)$,
then for any $j\geq 0$,
\begin{equation}\label{e:singlej}
\sum_{i=1}^{r+j} \epsilon_{(X_i+b_r)/a_r} \Rightarrow N=\sum_{i=1}^\infty
\epsilon_{g_\gamma^\leftarrow(\Gamma_i)} =\text{PRM}(m_\gamma),
\end{equation} in $M_+(\text{supp}_\gamma)$ and, in fact, jointly for any
$k\geq 0$,
\begin{equation}\label{e:joint}
\Bigl( \sum_{i=1}^{r+j} \epsilon_{(X_i +b_r)/a_r} ; 0\leq j\leq k \Bigr)
\Rightarrow (N,\dots,N)
\end{equation}
in $M_+(\text{supp}_\gamma) \times \dots \times M_+(\text{supp}_\gamma)$.
\end{Proposition}

\begin{proof}
We have \eqref{e:joint} following from \eqref{e:singlej} since with
respect to the vague distance $d(\cdot,\cdot)$ on 
$M_+(\text{supp}_\gamma)$ (see, eg. \cite[page 51]{resnickbook:2007})
$$d\Bigl(\sum_{i=1}^{r} \epsilon_{(X_i+b_r)/a_r} , \sum_{i=1}^{r+j}
\epsilon_{(X_i+b_r)/a_r} \Bigr) \Rightarrow 0
$$ for any $j\geq 0$. To verify this, let $f$ be positive and continuous with
compact support on $\text{supp}_\gamma$ and from equation 3.14 of 
\cite[page 51]{resnickbook:2007}, it suffices to show 
$$
 E \Bigl |\sum_{i=1}^{r} f\bigl((X_i+b_r)/a_r \bigr) -
 \sum_{i=1}^{r+j} f\bigl((X_i+b_r)/a_r\bigr) 
\Bigr |\to 0.$$
The difference is
\begin{align*}
 E  \sum_{i=r+1}^{r+j} f\bigl((X_i+b_r)/a_r \bigr) 
&=E \sum_{i=1}^j f\bigl((X_i+b_r)/a_r \bigr) \\
\intertext{and assuming the support of $f$ is a compact set $K$ in $\text{supp}_\gamma$,
  this is bounded above by}
&\sup_{x\geq 0} f(x) jP[X_1 \in a_rK-b_r] \to 0,
\end{align*} 
since for $x\in K$, $a_rx-b_r $ converges to the left endpoint of $F$
and under \eqref{e:assume}, there cannot be an atom at this left endpoint.

The result in \eqref{e:singlej} follows by a small modification of the
proof of Theorem 5.3 in \cite[page 138]{resnickbook:2007} since
\eqref{e:assume} is the same as \eqref{e:assumeF}.
\end{proof}

Now we turn to $\R^\infty$-convergence of the $r$th maximum sequence.
Continue to suppose \eqref{e:assume}.  Without
normalization, the sequence $\bMr$  
converges to a sequence all of whose entries  are the left endpoint of
$F$. In order to get $\bMr$ to converge, we must have
$M^{(r)}_r=\wedge_{i=1}^r X_i$
  converge and this helps explain why a domain of attraction condition for
  minima is relevant.  The  condition \eqref{e:assume} produces a non-trivial
limit. 

\begin{Proposition}\label{prop:MrConv}
Suppose 
the domain of attraction condition
\eqref{e:assume} holds.
Then in $\mathbb{R}^\infty $,
\begin{equation}\label{e:bigBoyConv}
\frac{\bMr +b_r}{a_r}=\Bigl(\frac{M_{r+j}^{(r)} +b_r}{a_r}, j\geq 0 \Bigr)\Rightarrow 
\Bigl(g_\gamma^\leftarrow ( \Gamma_l), l \geq 1\Bigr)\qquad (r\to\infty),
\end{equation}
where $\{\Gamma_l, l\geq 1\}$ are the points of a homogeneous Poisson
process on $\mathbb{R}_+$.
\end{Proposition}

\begin{proof}
Fix $j\geq 0$ and observe for $x\in \text{supp}_\gamma$,
\begin{align*}
\Bigl[ \frac{M^{(r)}_{r+j }+b_r}{a_r} >x\Bigr]=&
\Bigl[ \sum_{i=1}^{r+j} \epsilon_{(X_i+b_r)/a_r} (x,\infty)\geq r\Bigr] =
\Bigl[ \sum_{i=1}^{r+j} \epsilon_{(X_i +b_r)/a_r}( (-\infty,x]) \leq j\Bigr]\\
\intertext{and therefore}
\Bigl[ \frac{M^{(r)}_{r+j} +b_r}{a_r} \leq x\Bigr]=&\Bigl[ \sum_{i=1}^{r+j}
                                                \epsilon_{(X_i+b_r)/a_r}(
                                                (-\infty,x]) > j\Bigr].
\end{align*}
For a non-decreasing sequence $\{x_j\}$ of  real numbers in $\text{supp}_\gamma$,
\begin{align*}
P\Bigl\{ \bigcap_{j=0}^k   \Bigl[ \frac{M^{(r)}_{r+j} +b_r}{a_r} \leq &
  x_j\Bigr]\Bigr\}
=P\Bigl\{ \bigcap_{j=0}^k  \Bigl[ \sum_{i=1}^{r+j} \epsilon_{(X_i+b_r)/a_r}(
 [0,x_j]) > j\Bigr] \Bigr\}\\
\intertext{and applying \eqref{e:joint} yields}
\to & P\{ \bigcap_{j=0}^k  [ N((-\infty,x_j]) >j]\}
=P[\sum_{i=1}^\infty \epsilon_{g_\gamma^\leftarrow (\Gamma_i)} (-\infty,x_j]>j;
      \,j=0,\dots,k]\\
=&P[g_\gamma^\leftarrow (\Gamma_{j+1} )\leq x_j; j=0,\dots,k].
\end{align*}
This yields the announced result \eqref{e:bigBoyConv}.
\end{proof}

\section{Continuous time $r$th-order extremal   processes}\label{sec:cont_time}
  In this section we make the transition to continuous
time problems. The treatment is parallel to what we gave for discretely indexed
processes but here the processes are generated by two-dimensional Poisson
processes on $\R_+\times \R$ and correspond to $r$th order extremal
processes as $r\to\infty$. The continuous time case introduces a different
feature from the discrete index case;  namely, there are always
infinitely many values {\it below\/} your present position. This
necessitates differences in treatment, though both discrete and continuous time analyses rely on the presence of embedded Poisson processes. In continuous time we obtain modifications of Brownian motion limits whereas in discrete time
we obtain Poisson limits for the $r$th order extremes.  
  
  The setup is as follows.
For some numbers $-\infty \leq x_l<x_r \leq \infty$, 
and an infinite measure $\Pi$ on $(x_l,x_r)$ satisfying 
$\Pi (x_l,x_r)=\infty$ and $Q(x):=\Pi(x,x_r)<\infty$ for $x_l<x<x_r$,
let 
\beqq\label{e:prm}
N=\sum_k \epsilon_{(t_k,j_k)},\eeqq
be Poisson random measure on $[0,\infty)\times (x_l,x_r) $,
 with mean measure $\Leb\times \Pi$.
The notation
$\epsilon_{(t,x)}(\cdot) $ denotes a Dirac measure with mass $1$ at
the point $(t,x)$. Sometimes we write $(t_k,j_k) \in \supp (N)$ to
indicate the point $(t_k,j_k)$ is charged by $N$.
We assume $x_l $ and $x_r$ are
not atoms of $\Pi$ and in fact, results are most striking if we assume
$\Pi(\cdot)$ is atomless. (Otherwise, results would be stated in terms
of simplifications of point processes; see \cite{engelen:tommassen:vervaat:1988}.)
 Our
assumptions mean that 
\begin{enumerate}
\item The function $Q(x)$ satisfies $Q(x_r)=0$ and $Q(x_l)=\infty$
  so $Q:(x_l,x_r) \mapsto (0,\infty)$ and $Q(x)$ is non-increasing.
\item For any $t>0$ and $x_r\geq x> x_l: N \bigl( [0,t]\times
  (x,x_r)\bigr)<\infty$ almost surely.
\item For any $t>0$ and $x_r\geq x> x_l: N \bigl( [0,t]\times (x_l,x]
  \bigr)=\infty$ almost surely.
\end{enumerate}

Traditionally, the (first-order) extremal process is defined by
(\cite{
resnickbook:2008, deheuvels:1983, deheuvels:1982b,
dwass:1974,dwass:1966,dwass:1964,resnick:1975,resnick:rubinovitch:1973,
resnick:1974, shorrock:1975,weissman:1975b}),
$$Y(t)=Y^{(1)}(t)=\bigvee_{t_k\leq t} j_k, \quad 0<t<\infty,$$ 
the largest $j_k$ whose $t_k$ coordinate is at or before time $t$.
Alternatively we may write
$$Y(t)=\inf\{x>x_l: N\bigl( [0,t]\times (x,x_r)\bigr)=0\}=
 \inf\{x>x_l: N\bigl( [0,t]\times (x,x_r)\bigr)<1\}.$$
Here we investigate the analogue of Proposition \ref{prop:MrConv} as
$r\to\infty$ for the continuous time $r$th order extremal process 
$\bYr:= \{\Yr (t),0<t<\infty\}$ 
defined as,
\begin{equation}\label{def:Yr}
\Yr (t) :=\inf\{x>x_l: N\bigl( [0,t]\times (x,x_r)\bigr)<r\},\quad t>0.
\end{equation}
This means for $t>0$, $x_r\geq x>x_l$, 
$$[\Yr (t)>x]=[N\bigl( [0,t]\times (x,x_r)
\bigr)\geq r],$$ and therefore,
\begin{equation}\label{e:equiv}
[\Yr (t)\leq x]=[N\bigl( [0,t]\times (x,x_r)
\bigr) <r].
\end{equation}
Alternative ways of considering $\bYr$ are in \cite{engelen:tommassen:vervaat:1988}.

What is the behavior of $\{\bYr , r\geq 1\}$, considered as a sequence
of random elements of c\`adl\`ag space $D(x_l,x_r)$, as $r\to\infty$? This problem differs from the one considered in Section \ref{sec:rMax}
for $\bMr$. Unlike in  
Section \ref{sec:rMax}, there are always infinitely many points below
your current position and thus the left tail condition
\eqref{e:assumeF} used for $\bMr$ must be
different when considering $\bYr$. Analysis of the range of $\bYr$ is more
complicated and for the behavior of $\bYr$ itself, instead of relying on Poisson
behavior, we rely on asymptotic normality.

\subsection{The range $\mathcal{R}_r$ of $\bYr$.}
\label{subsec:rangeY}
Let $\mathcal{R}_r$ be the unique points in the set $\{Y^{(r)}(t),t>0\}$.
As in the discrete time case \eqref{e:rangeM}, we have
\beqq\label{e:rangeY}
\mathcal{R}_r= \bigcup_{p=1}^r\{j_k:(t_k,j_k)\in \supp(N),
\,N([0,t_k]\times [j_k,\infty))=p\}. \eeqq
To verify \eqref{e:rangeY} suppose $x\in \mathcal{R}_r$. There 
exists $t>0$ such that $Y^{(r)}(t) =x,$ and therefore there exists $(t_k,x)\in
\supp(N)$ such that $t_k\leq t$. If $N([0,t_k]\times [x,\infty))>r$,
then
$Y^{(r)}(t)>x$, giving a contradiction.  Thus $x$ is in the right side
of \eqref{e:rangeY}. Conversely, suppose $j_k$ satisfies that there exists $t_k$
such that $(t_k,j_k)\in \supp(N)$ and  $
N([0,t_k]\times [j_k,\infty))=p$ for some $p\leq r$. Then there exists
$t>t_k$ such that 
$N(t_k,t]]\times [j_k,\infty))=r-p$ and thus
$Y^{(r)}(t)=j_k$. Therefore, $j_k $ belongs to the left side of
\eqref{e:rangeY}. \halmos

When $\Pi$ is atomless, the range of $Y(t)=Y^{(1)}(t)$ is known to be
a Poisson process with mean measure determined by the monotone
function $S(x):=-\log \Pi(x,\infty), x>x_l$. This is discussed,
for example, in \cite[page 183]{resnickbook:2008}. In fact, from
\cite[Theorem 6.2, page 234]{engelen:tommassen:vervaat:1988}, the
$p$-records of $N$ are iid in $p$, and each sequence of $p$-records
forms PRM($S$). (A $p$-record of $N$ is a point $j_k$ such that there
exists $t_k$ making $(t_k,j_k)\in \supp(N)$ and $N([0,t_k]\times
[j_k,\infty))=p$.)  This and \eqref{e:rangeY} allow us to conclude
that $\mathcal{R}_r$ is a Poisson process with mean measure
$rS(\cdot)$. 
This achieves the continuous time analogue of the
discrete time discussion at the beginning of Subsection
\ref{sec:limRange},  and without any normalization we have
$$\mathcal{R}_r \Rightarrow \supp(S), \quad (r\to\infty),$$
in the Fell topology of closed subsets of $(x_l,x_r)$.

Paralleling the discrete time analysis, we next proceed by obtaining a
non-degenerate limit for $\mathcal{R}_r$.
We have to be more careful in the continuous case.
 The reason is that $\mathcal{R}_r$
is PRM with mean measure $rS(\cdot)$ and $S$ is Radon on $(x_l,x_r),$
and it may allocate infinite mass to a neighbourhood of both $x_l$ and
$x_r$. Recall $S(x):=-\log \Pi(x,x_l] $ satisfies $S(x_l)=-\infty$
and $S(x_r)=\infty$.

Assume without loss of generality that $x_l<0<x_r$. (If this is not
the case, choose an arbitrary point between $x_l$ and $x_r$.)  We make
a treatment parallel to the discrete one 
by splitting the Poisson points of $\mathcal{R}_r$ into those
above $0$ and those below. So write
$$\mathcal{R}_r=\mathcal{R}_r^+ \bigcup \mathcal{R}_r^-$$
where $\mathcal{R}_r^+ $ are the positive 
Poisson points of $\mathcal{R}_r$
and $\mathcal{R}_r^-$ are the negative points of $\mathcal{R}_r$.
 The two Poisson processes $\mathcal{R}_r^\pm$ are
independent because their points are in disjoint regions. Define
the two non-decreasing functions on $\R_+$,
\begin{align}
S^+(x)=&S(0,x]=S(x)-S(0),\quad &&0<x\leq x_r  \label{e:Sp}\\
S^-(x)=&S[-x,0)=S(0)-S(-x),\quad &&0<x \leq  -x_l . \label{e:Sm}
\end{align}
Assume there exist $a^\pm (t)>0,\,b^\pm (t) \in \R$ and infinite Radon measures
$S_\infty^\pm$ on $\R_+$ such that
such that as $t\to\infty$,
\begin{align}
&tS^+(a^+(t) x- b^+(t)) \to S_\infty^+ (x),\label{e:SpConv}\\
&tS^-(a^-(t) x- b^-(t)) \to S_\infty^- (x).\label{e:SmConv}
\end{align}

The form of $S_\infty^\pm$ is determined by defining  probability
distribution tails $\bar H^\pm (x)$ by
\begin{align}
&\bar H^+(x)=e^{-S^+(x)} , \quad 0<x<x_r,\label{e:Hp}\\
&\bar H^-(x)=e^{-S^-(x)} , \quad 0<x<-x_l.\label{e:Hm}
\end{align}
Note $\bar H^\pm (0)=e^{-S^\pm (0)} =e^{-0} =1$ and 
$\bar H^+ (x_r)=e^{-S^+ (x_r)} =e^{-\infty} =0$ and  $\bar
H^-(-x_l)=0$, similarly. Then, as in the discussion following
\eqref{e:assume}, we find for $\gamma^\pm \in \R$ that
$$e^{-S^\pm (x)} =G_{\gamma^\pm} (-x),$$ 
where $G_\gamma(x)$ has a form given after \eqref{e:max}. Note, if we
want 
$$a^+(t) = a^-(t) \quad \text{and} \quad b^+(t) = b^-(t)$$
up to convergence of types, we would need \cite{resnick:1971},
$-x_l=x_r$ and 
$$\bar H^+(x) \sim \bar H^-(x) \quad (x \to x_r ).$$

We now summarize.
\begin{Theorem}
The two Poisson processes $\mathcal{R}_r^\pm$ are independent with 
$\mathcal{R}_r =\mathcal{R}_r^+ \cup \mathcal{R}_r^-$  where
$\mathcal{R}_r^+$ has mean measure $rS^+$ on $\R_+$ and 
$-\mathcal{R}_r^-$ has mean measure $rS^-$ on $\R_+$ so that
$\mathcal{R}_r^-$ are points on $(-\infty,0).$ As $r\to\infty$, the
range centered and scaled converges to a limiting Poisson process,
$$\Bigl(
\frac{\mathcal{R}_r^+ +b^+(r)}{a^+(r)}, \frac{ -\mathcal{R}_r^+
  +b^-(r)}{a^-(r)} \Bigr)
\Rightarrow \Bigl(\mathcal{R}_\infty^+, -\mathcal{R}_\infty^-\Bigr),
$$
where the limits are independent Poisson processes on $\R_+$ with mean
measures $S_\infty^\pm$.  So if \eqref{e:SpConv} and \eqref{e:SmConv}
hold,
centering positive and negative range
points appropriately leads to a limiting Poisson process such that
positive points have mean measure $S^+_\infty(\cdot)$ and negative
range points made positive by taking absolute values have mean measure $S_\infty^-(\cdot).$
\end{Theorem}

\subsection{Finite dimensional convergence of $\bYr$ as random elements
  of $D(x_l,x_r)$.}\label{sec:fidiconvY}
In this subsection, we give a left-tail condition on $\Pi(\cdot)$
guaranteeing finite dimensional convergence of $\bYr$ to a transformed
Brownian motion.

Suppose there exist normalizing functions $a(r)>0, \; b(r)
\in \R$, and a non-decreasing limit function $h(x) \in \R$  with at least two
points of increase such that for $a(r)x+b(r) \in (x_l,x_r)$,
\begin{equation}\label{e:LTcondit}
\lim_{r\to\infty}
\frac{ r-Q\bigl(a(r)x+b(r) \bigr)}{\sqrt r}=h(x).
\end{equation}

Implications:
\begin{enumerate}
\item If we divide in \eqref{e:LTcondit} by $r$ instead of $\sqrt r$,
  the limit will be $0$ and therefore,
\begin{equation}\label{e:sim}
Q\bigl(a(r)x+b(r) \bigr) \sim r,\quad (r\to\infty).
\end{equation}
Therefore, since $r\to \infty$, we must have that
$Q\bigl(a(r)x+b(r)\bigr) \to \infty$ and
$(x_l,x_r)\ni a(r)x+b(r) \to x_l$.

\item For any $t>0$, 
\begin{align}
\frac{ r-tQ\bigl(a(r/t)x+b(r/t) \bigr)}{\sqrt r}=&
t\Bigl(
\frac{ r/t-Q\bigl(a(r/t)x+b(r) \bigr)}{\sqrt{ r/t} \sqrt t}  \Bigr)
                                                   \nonumber \\
\to&  \sqrt t  h(x), \quad (r\to\infty).\label{e:sqrt}\end{align}
\item If we write $r-Q=(\sqrt r - \sqrt Q)(\sqrt r + \sqrt Q)$ and use
  \eqref{e:sim}, we get
\begin{equation}\label{e:notQuite}
\sqrt r -\sqrt{ Q(a(r)x+b(r))} \to \frac 12 h(x).\end{equation}
Remember that $Q$ is decreasing and define a probability distribution
function $G(x)$ by $G(x):=\exp\{-\sqrt{Q(x)}\}$ so that $G$
concentrates on $(x_l,x_r)$. Then exponentiate in 
\eqref{e:notQuite} to get
$$e^{\sqrt r} e^{-\sqrt{Q(a(r)x+b(r))} } \to e^{\frac 12 h(x)},
\quad (r\to\infty)$$
or after a change of variables $s=e^{\sqrt r}$,
\begin{equation}
sG\bigl(a((\log s)^2)x+b((\log s)^2)\bigr)
=se^{-\sqrt{Q(a((\log s)^2)x+b((\log s)^2))} } \to e^{\frac 12 h(x)},
\end{equation}
as $s\to\infty$.
So we conclude that $G(x):=e^{-\sqrt{Q(x)}}$ is in a domain of
  attraction of an extreme value distribution for minima. This
  technique is essentially the same as the one used to study limit
  laws for record
  values in \cite{resnick:1973} or \cite{resnickbook:2008}.
\item Form of $h(x)$: As we saw following \eqref{e:assumeF}, if $\exp\{\frac{1}{2} h(x)\}$ plays the role of $g(x)$
then $h(x)$ must be of the form
$$e^{\frac 12 h(x)}=-\log G_\gamma (-x),$$
where $G_\gamma$ is an extreme value distribution for maxima of the
form
$$G_\gamma(x)=\exp\{ -(1+\gamma x)^{-1/\gamma}     \},\quad \gamma \in
\R, \,1+\gamma x>0.$$
So 
\beqq
\label{e:h}
\frac 12 h(x)= \begin{cases}
-\frac{1}{\gamma} \log (1-\gamma x),& \text{ if }\gamma \neq
0,\,1-\gamma x>0,\\
x,& \text{ if } \gamma=0,\,x\in\R.
\end{cases}
\eeqq
Observe that $h: \supp _\gamma \mapsto \R$ and 
$h^\leftarrow : \R \mapsto \supp_\gamma.$
Recalling the definition of $\text{supp}_\gamma$ from
\eqref{e:defsuppgamma}, we have
\begin{align*}
\text{supp}_\gamma=&\{x\in \R: 1-\gamma x>0\}\\
=&
\begin{cases}
(-\frac{1}{|\gamma |},\infty),& \text{ if }\gamma<0,\\
(-\infty,\frac{1}{|\gamma|},& \text{ if } \gamma >0,\\
\R, & \text{ if }\gamma=0.
\end{cases}
\end{align*}
\end{enumerate}

We apply these findings to obtain a marginal limit distribution for $Y^{(r)}(t)$ under the left tail  condition. Assume \eqref{e:LTcondit}.
We show that, for fixed $t$,  $\Yr (t)$
has a limit distribution as $r\to\infty$, after centering and norming.
This relies on an elementary fact: if $\{N_n \}$ is  a family of
Poisson random variables with $E(N_n)\to\infty$ then
\begin{equation}\label{e:an}
\frac{N_n -E(N_n)}{\sqrt{\text{Var}(N_n)}}\Rightarrow N(0,1),\quad
  (n \to \infty).\end{equation}

From \eqref{e:equiv}, we have
\begin{align*}
P\Bigl[&\frac{\Yr (t) -b(r/t)}{a(r/t)} \leq x \Bigr] =
P[ N([0,t]\times (a(r/t)x+b(r/t),\infty)) <r]\\
=&P[ \frac{N([0,t]\times (a(r/t)x+b(r/t),\infty)) -tQ(a(r/t)x+b(r/t))}
{\sqrt r}
 <\frac{r- tQ(a(r/t)x+b(r/t))}{\sqrt r}].
\end{align*}
From \eqref{e:sim}, $\sqrt r$ is asymptotic to the standard deviation
of the Poisson random variable and so the left side random variable
converges to a $N(0,1)$ random variable. Using \eqref{e:sqrt}, 
the right side converges to $\sqrt t h(x).$ We therefore conclude that
under the left tail condition \eqref{e:LTcondit}, for any fixed $t>0$,
\begin{equation}\label{e:onedim}
\lim_{r\to\infty} P\Bigl[\frac{\Yr (t) -b(r/t)}{a(r/t)} \leq x \Bigr] =
\Phi\bigl(\sqrt t h(x)\bigr),\quad x\in \supp_\gamma ,
\end{equation}
where $\Phi(x)$ is the standard normal cdf.

Now we can prove  the following finite dimensional convergence.

\begin{Proposition}\label{prop:bmlim}
Assume \eqref{e:LTcondit} holds with $h(x)$ given in \eqref{e:h}. Let $\{B(t),t\geq 0\}$ be standard
Brownian motion. Then as $r\to\infty$,
\begin{equation}\label{e:fidiconv}
\frac{\Yr (t) -b(r/t)}{a(r/t)} 
\Rightarrow
h^\leftarrow   \Bigl(
\frac{B(t)}{t}
\Bigr),
\end{equation}
in the sense of convergence of finite dimensional distributions for $t>0$.
\end{Proposition}

\begin{proof}
We illustrate the proof by showing bivariate
pairs converge for two values of $t$. So suppose $0<t_1<t_2$ and $x_1<x_2$ are in
$\supp_\gamma$ and we show as $r\to\infty$,
\begin{align}
 P\Bigl[\frac{\Yr (t_i) -b(r/t_i)}{a(r/t_i)} \leq x_i ;\,i=1,2\Bigr] 
&\to
 P\Bigl[h^\leftarrow \Bigl(\frac{B(t_i)}{t_i} \Bigr) \leq x_i
 ;\,i=1,2\Bigr]\nonumber \\
&=P\bigl[ B(t_i) \leq t_ih(x_i);\,i=1,2]. \label{e:pleaseShow}
\end{align}
We express the statements about $\bYr$ in terms of the Poisson
counting measure and consider:
\begin{align*}
&\begin{pmatrix}
N\bigl([0,t_1]\times  (a(r/t_1)x_1+b(r/t_1),\infty)\bigr)\\
N\bigl([0,t_2]\times (a(r/t_2)x_2+b(r/t_2),\infty)\bigr)
\end{pmatrix}\\
&\quad =
\begin{pmatrix}
N\bigl([0,t_1]\times (a(r/t_1)(x_1,x_2]+b(r/t_1),\infty)\bigr)
+N\bigl([0,t_1]\times (a(r/t_1)x_2+b(r/t_1),\infty)\bigr)\\
N\bigl([0,t_1]\times (a(r/t_2)x_2+b(r/t_2),\infty)\bigr)
+N\bigl((t_1,t_2]\times (a(r/t_2)x_2+b(r/t_2),\infty)\bigr)
\end{pmatrix}\\
&\quad= \begin{pmatrix}
N_1+N_2\\
N_3+N_4\end{pmatrix}.
\end{align*}
Consider the four terms $N_i, \,i=1,\dots,4$ in turn. 
\begin{enumerate}
\item The term $N_1$
appropriately normed converges to $0$,
\begin{equation}\label{e:oy}
\frac{N\bigl([0,t_1]\times
  (a(r/t_1)(x_1,x_2]+b(r/t_1),\infty)\bigr) -t_1\Pi(a(r/t_1)(x_1,x_2]+b(r/t_1))}{\sqrt r}
\Rightarrow 0.
\end{equation}
The reason is that the centering is 
\begin{align*}
\frac{t_1\Pi(a(r/t_1( x_1,x_2])}{\sqrt r}
=&\frac{t_1Q(ax_1+b)-t_1Q(ax_2+b)}{\sqrt r}\\
=&\frac{r-t_1Q(ax_2+b)}{\sqrt r}- \frac{r-t_1Q(ax_1+b)}{\sqrt r}\\
\to & \sqrt{t_1}(h(x_2)-h(x_1))>0.
\end{align*}
So the left side of  \eqref{e:oy} is of the form
$(N_r-\lambda_r)/\sqrt r$ where $\lambda_r/\sqrt r \to c>0$ and thus
$$\text{Var}\Bigl( (N_r-\lambda_r)/\sqrt r\Bigr)=\lambda_r/r\to 0,$$
which verifies the convergence to $0$ in \eqref{e:oy}.
\item The term $N_2$ becomes asymptotically normal. Let $Z_1$ be a
  standard normal random variable and apply \eqref{e:an} and
  \eqref{e:sim} to get
$$
\frac{N\bigl([0,t_1]\times (a(r/t_1)x_2+b(r/t_1),\infty)\bigr)
  -t_1Q(a(r/t_1)x_2+b(r/t_1)) }{\sqrt r}
 \Rightarrow  \sqrt{t_1} Z_1.
$$
\item For $N_3$, despite its dependence on the variable $t_2$, we also
  find
$$
\frac{N\bigl([0,t_1]\times (a(r/t_2)x_2+b(r/t_2),\infty)\bigr)
  -t_1Q(a(r/t_2)x_2+b(r/t_2)) }{\sqrt r}
 \Rightarrow  \sqrt{t_1} Z_1.
$$
This result uses a combination of the reasoning that was used for
$N_1,N_2$.
\item The term $N_4$ is independent of $N_1,N_2,N_3$ so there is a
  standard normal variable $Z_2\independent Z_1$ and 
$$\frac{N\bigl(
(t_1,t_2]\times (a(r/t_2)x_2+b(r/t_2),\infty)\bigr)
  -(t_2-t_1)Q(a(r/t_2)x_2+b(r/t_2)) }{\sqrt r}
 \Rightarrow  \sqrt{t_2-t_1} Z_2.
$$
\end{enumerate}

We conclude from this carving that
\begin{align*}
&
\begin{pmatrix}
\displaystyle{
\frac{N\bigl([0,t_1]\times  (a(r/t_1)x_1+b(r/t_1),\infty)\bigr)
  -t_1Q(a(r/t_1)x_1+b(r/t_1))}{\sqrt r} }\\
  \displaystyle{
\frac{N\bigl([0,t_2]\times
  (a(r/t_2)x_2+b(r/t_2),\infty)\bigr)-t_2Q(a(r/t_2)x_2+b(r/t_2))}{\sqrt
r}}
\end{pmatrix} \cr
&\cr
&\hskip2cm \Rightarrow
\begin{pmatrix}
\sqrt{t_1}Z_1\\
\sqrt{t_1}Z_1+\sqrt{t_2-t_1}Z_2
\end{pmatrix},
\end{align*}
as $r\to\infty$. 
Use \eqref{e:equiv} to write,
\begin{align*}
&P \Biggl[
\begin{pmatrix}
\displaystyle{\frac{\Yr(t_1)-a(r/t_1)}{b(r/t_1)}}\\
\displaystyle{\frac{\Yr(t_2)-a(r/t_2)}{b(r/t_2)}}\\
\end{pmatrix}
\leq  \begin{pmatrix}
x_1\\
x_2
\end{pmatrix}\Biggr]\\
&\cr
&\qquad=
P\Biggl[
\begin{pmatrix}
\displaystyle{\frac{N\bigl([0,t_1]\times  (a(r/t_1)x_1+b(r/t_1),\infty)\bigr)
  -t_1Q(a(r/t_1)x_1+b(r/t_1))}{\sqrt r}}\\
\displaystyle{\frac{N\bigl([0,t_2]\times
  (a(r/t_2)x_2+b(r/t_2),\infty)\bigr)-t_2Q(a(r/t_2)x_2+b(r/t_2))}{\sqrt
r}}
\end{pmatrix}\cr
&\cr
&\qquad <
\begin{pmatrix}
\displaystyle{\frac{r-t_1Q(a(r/t_1)x_1+b(r/t_1))}{\sqrt r}}\\
\displaystyle{\frac{r-t_2Q(a(r/t_2)x_2+b(r/t_2))}{\sqrt r}}\\
\end{pmatrix}\\
&\qquad
 \to P[\sqrt{ t_1}Z_1 \leq t_1h(x_1),
  \sqrt{t_1}Z_1+\sqrt{t_2-t_1}Z_2 \leq t_2 h(x_2)]\ {\rm (as}\ r\to\infty)\\
&\qquad
=P\bigl[\frac{B(t_1)}{t_1}\leq h(x_1),\frac{B(t_2)}{t_2}\leq h(x_2)\bigr]\\
&\qquad
=P\bigl[h^\leftarrow \bigl(\frac{B(t_1)}{t_1} \bigr)\leq x_1,h^\leftarrow \bigl(\frac{B(t_2)}{t_2}
  \bigr)\leq x_2\bigr].
\end{align*}
This verifies \eqref{e:pleaseShow}.
\end{proof}

\section{Final thoughts}\label{sec:final}
The results of this paper suggest some obvious questions the answers
to which have so far eluded us. Is there a jump process 
limit -- presumably some sort of extremal process -- in \eqref{e:fidiconv}
corresponding to some sort of Poisson limit regime as opposed to the
Brownian motion limit regime? In Proposition \ref{prop:bmlim} is a
stronger form of convergence -- say in the $J_1$-topology -- possible? And
so far, the mathematics of proving in a nice
way that $\{\bYr, r\geq 1\}$ is Markov
in the {\cadlag}  space $D(0,\infty)$  has not cooperated.

\bibliography{bibfile}

\def\cprime{$'$}
\begin{thebibliography}{10}

\bibitem{arnold:becker:gather:zahedi:1984}
B.~C. Arnold, A.~Becker, U.~Gather, and H.~Zahedi.
\newblock On the {M}arkov property of order statistics.
\newblock {\em J. Statist. Plann. Inference}, 9(2):147--154, 1984.

\bibitem{cramer:tran:2009}
E.~Cramer and T.H. Tran.
\newblock Generalized order statistics from arbitrary distributions and the
  {M}arkov chain property.
\newblock {\em J. Statist. Plann. Inference}, 139(12):4064--4071, 2009.

\bibitem{dehaan:ferreira:2006}
L.~de~Haan and A.~Ferreira.
\newblock {\em Extreme Value Theory: An Introduction}.
\newblock Springer-Verlag, New York, 2006.

\bibitem{deheuvels:1982b}
P.~Deheuvels.
\newblock A construction of extremal processes.
\newblock In {\em Probability and {S}tatistical {I}nference (Bad Tatzmannsdorf,
  1981)}, pages 53--57. Reidel, Dordrecht, 1982.

\bibitem{deheuvels:1983}
P.~Deheuvels.
\newblock The strong approximation of extremal processes. {II}.
\newblock {\em Z. Wahrscheinlichkeitstheorie und Verw. Gebiete}, 62(1):7--15,
  1983.

\bibitem{dwass:1964}
M.~Dwass.
\newblock Extremal processes.
\newblock {\em Ann. Math. Statist}, 35:1718--1725, 1964.

\bibitem{dwass:1966}
M.~Dwass.
\newblock Extremal processes. {II}.
\newblock {\em Illinois J. Math.}, 10:381--391, 1966.

\bibitem{dwass:1974}
M.~Dwass.
\newblock Extremal processes. {III}.
\newblock {\em Bull. Inst. Math. Acad. Sinica}, 2:255--265, 1974.
\newblock Collection of articles in celebration of the sixtieth birthday of Ky
  Fan.

\bibitem{engelen:tommassen:vervaat:1988}
R.~Engelen, P.~Tommassen, and W.~Vervaat.
\newblock Ignatov's theorem: a new and short proof.
\newblock {\em J. Appl. Probab.}, Special Vol. 25A:229--236, 1988.
\newblock A celebration of applied probability.

\bibitem{goldie:rogers:1984}
C.~M. Goldie and L.~C.~G. Rogers.
\newblock The {$k$}-record processes are i.i.d.
\newblock {\em Z. Wahrsch. Verw. Gebiete}, 67(2):197--211, 1984.

\bibitem{goldie:maller:1999}
C.M. Goldie and R.A. Maller.
\newblock Generalized densities of order statistics.
\newblock {\em Statist. Neerlandica}, 53(2):222--246, 1999.

\bibitem{ignatov:1977}
Z.~Ignatov.
\newblock Ein von der {V}ariationsreihe erzeugter {P}oissonscher
  {P}unktproze\ss.
\newblock {\em Annuaire Univ. Sofia Fac. Math. M\'ec.}, 71(2):79--94 (1986),
  1976/77.

\bibitem{matheron:1975}
G.~Matheron.
\newblock {\em Random {S}ets and {I}ntegral {G}eometry}.
\newblock John Wiley\thinspace \&\thinspace Sons, New York-London-Sydney, 1975.
\newblock With a foreword by G.S. Watson, Wiley Series in Probability and
  Mathematical Statistics.

\bibitem{molchanov:2005}
I.~Molchanov.
\newblock {\em Theory of Random Sets}.
\newblock Probability and its Applications (New York). Springer-Verlag London
  Ltd., London, 2005.

\bibitem{renyi:1962}
A.~R{\'e}nyi.
\newblock Th\'eorie des \'el\'ements saillants d'une suite d'observations.
\newblock {\em Ann. Fac. Sci. Univ. Clermont-Ferrand No.}, 8:7--13, 1962.

\bibitem{resnick:1971}
S.I. Resnick.
\newblock Tail equivalence and its applications.
\newblock {\em J. Appl. Probab.}, 8:136--156, 1971.

\bibitem{resnick:1973}
S.I. Resnick.
\newblock Limit laws for record values.
\newblock {\em Stochastic Processes Appl.}, 1:67--82, 1973.

\bibitem{resnick:1974}
S.I. Resnick.
\newblock Inverses of extremal processes.
\newblock {\em Adv. in Appl. Probab.}, 6:392--406, 1974.

\bibitem{resnick:1975}
S.I. Resnick.
\newblock Weak convergence to extremal processes.
\newblock {\em Ann. Probab.}, 3(6):951--960, 1975.

\bibitem{resnickbook:2007}
S.I. Resnick.
\newblock {\em Heavy Tail Phenomena: Probabilistic and Statistical Modeling}.
\newblock Springer Series in Operations Research and Financial Engineering.
  Springer-Verlag, New York, 2007.
\newblock ISBN: 0-387-24272-4.

\bibitem{resnickbook:2008}
S.I. Resnick.
\newblock {\em Extreme Values, Regular Variation and Point Processes}.
\newblock Springer, New York, 2008.
\newblock Reprint of the 1987 original.

\bibitem{resnick:rubinovitch:1973}
S.I. Resnick and M.~Rubinovitch.
\newblock The structure of extremal processes.
\newblock {\em Adv. in Appl. Probability}, 5:287--307, 1973.

\bibitem{ruschendorf:1985}
L.~R{\"u}schendorf.
\newblock Two remarks on order statistics.
\newblock {\em J. Statist. Plann. Inference}, 11(1):71--74, 1985.

\bibitem{shorrock:1974}
R.~W. Shorrock.
\newblock On discrete time extremal processes.
\newblock {\em Adv. in Appl. Probab.}, 6:580--592, 1974.

\bibitem{shorrock:1975}
R.~W. Shorrock.
\newblock Extremal processes and random measures.
\newblock {\em J. Appl. Probability}, 12:316--323, 1975.

\bibitem{stam:1985}
A.~J. Stam.
\newblock Independent {P}oisson processes generated by record values and
  inter-record times.
\newblock {\em Stochastic Process. Appl.}, 19(2):315--325, 1985.

\bibitem{vervaat:holwerda:1997}
W.~Vervaat and H.~Holwerda, editors.
\newblock {\em Probability and lattices}, volume 110 of {\em CWI Tract}.
\newblock Stichting Mathematisch Centrum, Centrum voor Wiskunde en Informatica,
  Amsterdam, 1997.

\bibitem{weissman:1975b}
I.~Weissman.
\newblock Extremal processes generated by independent nonidentically
  distributed random variables.
\newblock {\em Ann. Probab.}, 3:172--177, 1975.

\end{thebibliography}
\end{document}